\documentclass{article}

\usepackage[paper=a4paper,dvips,top=3cm,left=3cm,right=3cm,
   foot=1.4cm,bottom=3cm]{geometry}
\usepackage{times,scrtime}
\usepackage{graphicx}
\usepackage{amsfonts, amsmath}
\usepackage{amsthm}
\usepackage[mathcal]{euscript}      
\usepackage{bm}                     
\usepackage{algorithm, algorithmic} 
\usepackage{color}
\usepackage{amssymb}
\usepackage{mathrsfs}
\usepackage{abstract}
\usepackage{xcolor}
\usepackage{tabulary}
\usepackage{array,multirow}
\usepackage{url}
\usepackage[numbers]{natbib}
\usepackage{verbatim}

\numberwithin{equation}{section} 

\newtheorem{Definition}{Definition}[section]
\newtheorem{Lemma}{Lemma}[section]
\newtheorem{Proposition}{Proposition}[section]
\newtheorem{Theorem}{Theorem}[section]
\newcommand{\vt}[1]{{\bf #1}}
\newcommand{\x}{\vt{x}}
\newcommand{\w}{\vt{w}}

\newcommand{\y}{\vt{y}}
\newcommand{\p}{\vt{p}}
\newcommand{\dd}{\vt{d}}

\newcommand{\Ten}{\mathcal{T}}
\newcommand{\REAL}{\mathbb{R}}
\newcommand{\Adj}{\mathcal{A}}
\newcommand{\R}{\mathbb{R}}
\newcommand{\g}{\nabla f}
\newcommand{\T}{\top}
{\par\addvspace{\baselineskip}}
{\par\addvspace{\baselineskip}}

\title{Computing the $p$-Spectral Radii of Uniform Hypergraphs with Applications}

\author{
  {Jingya Chang}\thanks{%
    School of Mathematics and Statistics, Zhengzhou University, Zhengzhou 450001, China and
    Department of Applied Mathematics, The Hong Kong Polytechnic University,
    Hung Hom, Kowloon, Hong Kong
    ({\tt jychang@zzu.edu.cn}). This author's work was partially supported by the National Natural Science Foundation of China (grant No. 11401539 and 11571178)}
    \and
{Weiyang Ding}\thanks{ Department of Applied Mathematics, The Hong Kong Polytechnic University,
    Hung Hom, Kowloon, Hong Kong ({\tt weiyang.ding@gmail.com}). This author's work was partially supported by the Hong Kong Research Grant Council
    (Grant No. C1007-15G).}
  \and
  {Liqun Qi}\thanks{%
    Department of Applied Mathematics, The Hong Kong Polytechnic University,
    Hung Hom, Kowloon, Hong Kong ({\tt maqilq@polyu.edu.hk}).
    This author's work was partially supported by the Hong Kong Research Grant Council
    (Grant No. PolyU 501913, 15302114, 15300715, 15301716 and C1007-15G).}
    \and
      {Hong Yan}\thanks{%
  Department of Electronic Engineering, City University of Hong Kong, Kowloon, Hong Kong  ({\tt h.yan@cityu.edu.hk}).
  This author's work was partially supported by the Hong Kong Research Grants Council (Grant No. C1007-15G).}
}

\begin{document}

\maketitle

\begin{abstract}

The $p$-spectral radius of a uniform hypergraph covers many important concepts,
such as Lagrangian and spectral radius of the hypergraph,
and is crucial for solving spectral extremal problems of hypergraphs.
In this paper, we establish a spherically constrained maximization model and
propose a first-order conjugate gradient algorithm to
compute the $p$-spectral radius of a uniform hypergraph (CSRH).
By the semialgebraic nature of the adjacency tensor of a uniform hypergraph,
CSRH is globally convergent and obtains the global maximizer with a high probability.
When computing the spectral radius of the adjacency tensor of a uniform hypergraph,
CSRH stands out among  existing approaches.
Furthermore, CSRH is competent to calculate the $p$-spectral radius of a hypergraph with millions of vertices
and to approximate the Lagrangian of a hypergraph.
Finally, we show that the CSRH method is capable of ranking real-world data set
based on solutions generated by the $p$-spectral radius model.
\end{abstract}



\pagestyle{myheadings}
\thispagestyle{plain}
\markboth{JINGYA CHANG, WEIYANG DING, AND LIQUN QI}{COMPUTING P-SPECTRAL RADII OF HYPERGRAPHS}

\section{Introduction}

With the emergence of big data  in various field of our social life, it becomes  significant and challenging to analyze the massive data  and extract valuable information from them. Hypergraph, as an extension of graph, provides an efficient way to represent complex relationships among objects  in applied science, such as chemistry \cite{konstantinova2001application,klamt2009hypergraphs}, computer science \cite{gunopulos1997data,pliakos2015weight,karypis1999multilevel}, and image processing \cite{bretto2005hypergraph,ducournau2012reductive,chen2016fiber}. The spectral hypergraph theory has been widely studied in \cite{cooper2012spectra,kang2015p,keevash2014spectral,lu2016connected,pearson2014spectral,xie2013z,yue2016largest}, which reveal   combinatorial and geometric structures of hypergraphs.  Moreover, spectral hypergraph approaches are useful tools to  address issues in real world. Spectral hypergraph  partitioning and spectral hypergraph  clustering have broad applications in network analysis \cite{michoel2012alignment,rohe2011spectral}, image segmentation \cite{ducournau2009multilevel}, multi-label classification \cite{sun2008hypergraph}, machine learning \cite{zhou2006learning}, and data analysis \cite{agarwal2005beyond,li2014context}. Hypergraph spectral hashing techniques highly contribute to problems of similarity search and retrieval of social image \cite{zhuang2011hypergraph,liu2013hypergraph}.

In this paper, we focus on the computation of  $p$-spectral radii of uniform hypergraphs.
The $p$-spectral radius of a hypergraph was introduced in \cite{keevash2014spectral} and linked with extremal hypergraph problems.
Extremal graph theory, as a branch of graph theory, is one of the most  attractive and best studied area in combinatorics.
Tur\'{a}n \cite{turan1941eine} introduced the famous Tur\'{a}n graph   and  Tur\'{a}n  theorem in 1941, when  is regarded as the start of the extremal graph theory. Naturally, the question was extended from graph to hypergraph in \cite{turan1961research} that is to find the largest number of edges in a hypergraph which is $F$-free\footnote{ A uniform hypergraph that does not have a subgraph isomorphic to the uniform  hypergraph $F$ is said to be $F$-free.}. Although the Tur\'{a}n-type problem is adequately complete for ordinary graphs, cases are much more challenging when it comes to hypergraph.
In \cite{nikiforov2007bounds}, Nikiforov proved the spectral Tur\'{a}n-type inequality
which generalized the Tur\'{a}n theorem.
In \cite{keevash2014spectral}, the $p$-spectral  version of Nikiforov's inequality  and the $p$-spectral  version of a hypergraph Tur\'{a}n result were given, and it was showed that this  result can be employed in solving `degenerate' Tur\'{a}n-type problems. Furthermore, it was proved that the edge extremal problems  are asymptotically equivalent to the extremal $p$-spectral radius problems  in \cite{nikiforov2013analytic}.


The $p$-spectral radius of a hypergraph covers not only the number of edges in extremal problems, but also  the notions, such as  Lagrangian, and  the spectral radius of a hypergraph \cite{lu2016connected}.
When  $p=1$, the $p$-spectral radius of a hypergraph turns out to be its Lagrangian. The Lagrangians of graph and hypergraph were proposed in \cite{motzkin1965maxima} to prove the Tur\'{a}n's theorem for graphs. The Largrangians of hypergraphs were used   to disprove the conjecture of Erd\"{o}s  \cite{erdos1946structure,frankl1984hypergraphs} and to  find non-jumping numbers for hypergraphs \cite{frankl2007note,peng2008using,peng2008generating}.  Also, the Lagrangian of a hypergraph is associated with problems of determining Tur\'{a}n densities of hypergraphs \cite{brown1984digraph,keevash2011hypergraph,mubayi2006hypergraph,sidorenko1987maximal}, which is an asymptotic solution to a (non-degenerate) Tur\'{a}n problem.
When $p=2$, the $p$-spectral radius of a uniform hypergraph  is the largest Z-eigenvalue \cite{qi2005eigenvalues} of its adjacency tensor. When $p$ is even and equals the order of this hypergraph, the $p$-spectral radius  becomes the largest  H-eigenvalue  of the adjacency tensor of $G$. Therefore, the $p$-spectral radius is connected with the (adjacency) spectral radius of a hypergraph \cite{hu2015laplacian,li2015extremal,lu2016connected}. Additionally, \citeauthor{kang2015p} provided solutions to several $p$-spectral radius related extremal problems in \cite{kang2015p}. Nikiforov in \cite{nikiforov2014analytic} did a comprehensive study  and obtained many  theoretical conclusions about $p$-spectral radius.

Apart from the application in extremal hypergraph theory, the $p$-spectral radius model constructs a framework to quantify the importance of objects or centrality in networks. Evaluating the significance or popularity of objects is a significant problem in data mining. It can be used to determine the importance of web pages \cite{page1999pagerank,kolda2005higher,ding2002pagerank}, forecast customer behaviour \cite{krohn2012multi}, retrieve images \cite{huang2010image} and so on. In  the $p$-spectral radius model, entries of the vector associated with the $p$-spectral radius of a hypergraph are called $p$-optimal weighting and represent the significance of its corresponding vertices. The ranking result varies when $p$  changes. We will explain the meaning of different ranking results and show the numerical performance of our algorithm in sorting real-life data in Section 6.

Calculation of $p$-spectral radii of hypergraphs is related to several methods for evaluating tensor eigenvalues.  Algorithms for tensor eigenvalues, such as the shifted symmetric higher-order power method (SS-HOPM) in \cite{kolda2011shifted}, the generalized eigenproblem adaptive power (GEAP) method in \cite{kolda2014adaptive}, an extension of Collatz's method (NQZ)  in \cite{ng2009finding}, and the CEST  method, can be  employed when $p$ equals 2 or when $p$ equals the  order of an even-uniform hypergraph. When $p$ is even, the $p$-spectral radius problem is equivalent to the  generalized tensor eigenvalue problem  \cite{chang2008perron, ding2015fast}. Therefore,  methods for this generalized tensor eigenvalue problem, such as the polynomial optimization related algorithm for finding all real eigenvalues of a symmetric tensor given by \citeauthor{cui2014all} in \cite{cui2014all}, and the  homotopy approach for all eigenpairs of general real or complex tensors proposed by \citeauthor{chen2016computing} in \cite{chen2016computing} can be employed to compute even $p$-spectral radius of  small scale hypergraphs.
However,  the problem of computing $p$-spectral radii of  arbitrary hypergraph is still open. This is the main motivation of our paper.

  To solve the $p$-spectral radius problem, we introduce a spherically constrained maximization model, which is equivalent to the original problem. Then we use an effective conjugate gradient method  to acquire an ascent direction for the constrained optimization model. Next, we employ the Cayley transform to project the ascent direction on the unit sphere. It is proved that there exists a positive parameter in the curvilinear line search such that the Wolfe conditions hold. Based on the above foundation, we propose a numerical method for computing $p$-spectral radii of hypergraphs (CSRH) with $p>1.$  When $p=1,$ the CSRH method is able to approximate the $1$-spectral radii (Largrangians) of hypergraphs. In the convergence analysis, we prove that the CSRH algorithm is  convergent and  it converges to the global optimization point with high probability.  Numerical experiments show that CSRH  is preponderant when compared to  existing methods  for computing Z-eigenvalues and H-eigenvalues of adjacency tensors.  Moreover, CSRH  is capable of calculating $p$-spectral radii of hypergraphs with millions of vertices effectively. In addition, we find that the significance of vertices of hypergraphs is related to the order of elements of the $p$-optimal weighting. Therefore, we apply the CSRH method to rank the vertices of the corresponding hypergraph  from different viewpoints when $p$ is different, which is useful in network analysis. As an example,  we show that our numerical results  agree with the observed data of a small weighted hypergraph. Furthermore, we successfully rank 10305 authors based on their  publication information  by establishing a hypergraph model and  using CSRH  to solve the corresponding $p$-spectral radius problem. We sort the authors from the view of individual and group respectively. The result of our ranking can be reasonably explained and are in line with the existing consequences in \cite{ng2011multirank}.

%

The paper is organized as follows. In Section $2$, we introduce mathematical notions. The computational issues about $p$-spectral radius are addressed in Section $3$, where our new method  CSRH for computing $p$-spectral radii of hypergraphs is given. In Section $4$, we analyze the convergent property of the CSRH method. The  numerical experiments are represented in Section $5.$ In Section $6,$  we show the application of CSRH method in network analysis. The ranking results of a toy example and  a large scale real-world problem are presented. Finally, we draw  conclusions in Section $7$.

\section{Preliminary}

In this section we introduce useful notions and  important results on hypergraphs and tensors.
Let $\R^{[r,n]}$ be the $r$th order $n$-dimensional real-valued tensor space, i.e.,
$$\R^{[r,n]}\equiv \R^{\overbrace{{n \times n \times \cdots \times n}}^{r\text{-times}}}.$$
A  tensor $\Ten=(t_{i_1 \cdots i_r}) \in\REAL^{[r,n]}$  with $ i_j=1,\ldots,n$ for  $j=1,\ldots,r,$ is said to be symmetric, if $t_{i_1 \cdots i_r}$ is unchanged under any permutation of  indices \cite{chen2016positive}. Two operations  between
$\Ten$ and any vector $\x \in \mathbb{R}^n$ are defined as
\begin{equation*}
    \Ten\x^r \equiv \sum_{i_1=1}^n\cdots\sum_{i_r=1}^n
      t_{i_1\cdots i_r}\x_{i_1}\cdots \x_{i_r}
\end{equation*}
and \begin{equation*}
    (\Ten\x^{r-1})_i \equiv \sum_{i_2=1}^n\cdots\sum_{i_r=1}^n
      t_{ii_2\cdots i_r}\x_{i_2}\cdots \x_{i_r}, \quad \, \text{for} \quad i=1,\ldots,n.
\end{equation*}
Note that,  $\Ten\x^r \in \mathbb{R}$ and $\Ten\x^{r-1} \in \mathbb{R}^n$  are a scalar and a vector respectively,  and $\Ten\x^r = \x^
{\T} (\Ten\x^{r-1}).$

If there exists a real number $\lambda$ and a nonzero real vector $\x$ satisfying
\begin{equation}\label{H-eigenvalue}
  \Ten \x^{m-1}=\lambda \x^{[m-1]},
\end{equation}
then $\lambda$ is called an H-eigenvalue of $\Ten$ with $\x$ being the associated H-eigenvector \cite{qi2005eigenvalues,qi2016tensor}. Additionally,   $\x^{[m-1]} \in \mathbb{R}^n$ is a vector, of which the $i$th element is $\x_i^{m-1}.$
When a real vector $\x$ and a real number $\lambda$ satisfy the following system
\begin{eqnarray*}\label{Z-eignvalue}
\left\{\begin{aligned}
   \Ten \x^{m-1} &=& \lambda \x \\
  \x^{\T} \x&=& 1,
  \end{aligned}\right.
\end{eqnarray*}
$\lambda$ is called a Z-eigenvalue of $\Ten$ and $\x$ is the corresponding Z-eigenvector \cite{qi2005eigenvalues}.

 \begin{Definition}[Hypergraph]
A  hypergraph is defined as $G=(V,E)$, where $V=\{1,2,\ldots,n\}$ is the vertex set and
  $E=\{e_1,e_2,\ldots,e_m\}\subseteq 2^{V}$ (the powerset of $V$) is the edge set. We call $G$ an $r$-uniform hypergraph when  $|e_p|=r \geq 2$ for $p=1,\ldots,m$ and $e_i \neq e_j$ in case of $i \neq j.$

  If each edge of a hypergraph is linked with a positive number $s(e),$ then this hyperpragh is called a weighted hypergraph and $s(e)$ is the weight associated with the edge $e.$   An ordinary hypergraph  can be regarded as a weighted hypergraph with the weight of  each edge being $1.$
\end{Definition}
 In the rest of this paper, an $r$-uniform hypergraph is abbreviated to an $r$-graph for convenience and hence the hypergraph $G$  refers to an $r$-graph.
  The degree of a vertex $i\in V$ is given by $d(i)=\mathrm{sum} \{s(e): i\in e, e\in E\}.$  The \textit{weight polynomial} of $G$ \cite{talbot2002lagrangians} is defined as
 \begin{equation}\label{weightPoly}
   w(G,\x)=\sum_{ e=\{i_1, \ldots , i_r\}  \in E} s(e)\, \x_{i_1} \cdots \x_{i_r},
 \end{equation}
 in which $\x$ is a vector in $\mathbb{R}^n$, $e=\{i_1, \ldots , i_r\}$ is an edge of $G$ and $s(e)$ is the weight of $e.$

 \begin{Definition}[$p$-spectral radius \cite{keevash2014spectral,kang2015p}]
   When $p \geq 1$, the $p$-spectral radius of  $G$, denoted by $\lambda^{(p)}(G)$, is defined  as
\begin{equation}\label{Origi_pSpec}
\lambda^{(p)}(G)= r!\max _{\|\x\|_p=1}  w(G,\x),
\end{equation}
and we call any vector $\x$ solving \eqref{Origi_pSpec} a $p$-optimal weighting of $G$ \cite{caraceni2011lagrangians}.
 \end{Definition}
 When $p=1,$ the $p$-spectral radius of $G$ coincides with its Lagrangian $\lambda_L(G)$ \cite{frankl1989extremal,talbot2002lagrangians}, which is defined as
\begin{equation}\label{lagrangian}
\lambda_L(G)=
   \left\{
   \begin{array}{ll}
     \max & \hbox{$w(G,\x)$} \\
     \mathrm{s.t.} & \hbox{$\sum_{i=1}^{r} \x_i  =  1,$} \\
     & \hbox{$\x_i  \geq 0,\quad \text{for} \,\quad i=1,\ldots, r.$}
   \end{array}
 \right.
\end{equation}
 The vector $\x$ related to the Lagrangian of $G$ is named the optimal legal weighting \cite{caraceni2011lagrangians,talbot2002lagrangians}.
\begin{Definition}[Adjacency tensor ]\label{Adjacency tensor}
 The adjacency tensor $\Adj$ of a weighted $r$-graph $G$ is defined as
  an $r$th order $n$-dimensional symmetric tensor with its elements being
  \begin{equation*}
    a_{i_1 \cdots i_r}=\left\{\begin{aligned}
      & \frac{s(e)}{(r-1)!} && \quad \emph{if } \,\{i_1,\ldots,i_r\}\in E, \\
      &0                && \quad \emph{ otherwise. }
    \end{aligned}\right.
  \end{equation*}
 \end{Definition}
It is obvious  from \eqref{Origi_pSpec} that  the
$2$-spectral radius is exactly the product of $(r-1)!$ times the largest  Z-eigenvalue of the adjacency tensor $\mathcal{A}$, and when
 $r$ is even the $r$-spectral radius is $(r-1)!$ times the largest H-eigenvalue of
$\mathcal{A}$  \cite{qi2005eigenvalues}. 

Although there is no general formula or algorithm for us to compute the $p$-spectral radius of a hypergraph directly, research on $p$-spectral radius of hypergraphs with  certain structures has made some progress.

\begin{Theorem}[\cite{nikiforov2014analytic}]\label{BetaStarTheor}
  Let $r$-graph $G$ be a $\beta$-star with $m$ edges .\\[0.2cm]
     a. If $p>r-1,$ then $\lambda^{(p)}(G)=r!r^{-\frac{r}{p}}m^{(1-\frac{r-1}{p})}.$ \\[0.1cm]
  b. If $p<r-1,$ then $\lambda^{(p)}(G)=r!r^{-\frac{r}{p}}.$\\[0.1cm]
   c. If $p=r-1,$ then $\lambda^{(p)}(G)=(r-1)!r^{-\frac{1}{r-1}}.$

\end{Theorem}

\begin{Proposition}[\cite{caraceni2011lagrangians}]\label{PropLagran}
  If $G$ is a complete $r$-graph with $n$ vertices, then the Lagrangian of $G$ is
  \begin{equation}
   \lambda_{L}(G)=\biggl(\begin{array}{c}
                           n \\
                           r
                        \end{array}\biggr)\frac{1}{n^r}.
    \end{equation}
\end{Proposition}
A multiset is an extension of the ordinary set, such that the objects or elements in the multiset are repeatable. If the edge set $E$ of a hypergraph $G$ is a set of multisets, then $G$ is called a multi-hypergraph \cite{pearson2014spectral}. Naturally, the $p$-spectral radius problem can be extended from hypergraph to muli-hypergraph. The algorithm and theoretical analysis in the following part of this paper are  also applicable to $p$-spectral radius problems of multi-hypergraphs.
In the rest of this paper, the symbol $\|\cdot\|$ refers to $\ell_2$ norm and the parameter $p$ is a positive integer unless stated otherwise.

%

\section{Computation of the $p$-spectral radius of a hypergraph}
 We transform the $p$-spectral radius in \eqref{Origi_pSpec} into a spherically constraint optimization problem and propose an iterative algorithm to solve it.

\subsection{Spherically constraint form for  $\lambda^{(p)}(G)$}

The $p$-spectral radius of  $G$ in \eqref{Origi_pSpec} can be reformulated as
\begin{equation}\label{original model}
\lambda^{(p)}(G)= \max _{\|\x\|_p=1} (r-1)! \Adj \x^r
\end{equation}
where $\Adj$ is the adjacency tensor of $G.$
The maximization problem \eqref{original model} is equivalent to an unconstrained format, that is
\begin{equation}\label{unconstrained}
\lambda^{(p)}(G)= \max_{ \x\neq 0} (r-1)! \frac{\mathcal{A} \x^r}{\|\x\|_p^{r}}.
\end{equation}
In order to restrict the search region and keep the vector $\x$ away from zero,
we add a spherically constraint on  $\lambda^{(p)}(G) $ in \eqref{unconstrained}. Due to the  zero-order homogeneous property of $\mathcal{A} \x^r/\|\x\|_p^{r},$   we can obtain  $\lambda^{(p)}(G) $    by solving the following  problem
\begin{equation}\label{sperical constrained model}
\left\{
\begin{aligned}
 \max f(\x) &=(r-1)!\frac{\mathcal{A} \x^r}{\|\x\|_p^{r}} \\
\text{s.t.} \,\, \|\x\|_2& =1.
\end{aligned}
\right.
\end{equation}
When $p>1$,  the  objective function $f(\x)$ is differentiable for any nonzero $\x$ and the  gradient of $f(\x)$ is
\begin{equation}\label{gradient}
\g(\x)=\frac{r!}{\|\x \|_p^r} \left(\mathcal{A} \x^{r-1}- \mathcal{A} \x^r \|\x\|_p^{-p} \x^{\langle p-1 \rangle}\right),
\end{equation}
where $\x^{\langle p-1 \rangle}$ represents a vector whose $i$th element is $(\x^{\langle p-1 \rangle})_i=|x_i|^{p-1} \text{sgn} (x_i).$
Since $f(\x)$ is zero-order homogeneous, we have
\begin{equation}\label{xTg(x)}
\x^\T \g(\x)=0
\end{equation}
 for any $0 \neq \x \in \mathbb{R}^n.$

Based on the spherically constrained form in \eqref{sperical constrained model}, we have the following proposition, which provides a way to approximate the $p$-spectral radius of a hypergraph when it cannot be computed directly.
 \begin{Proposition}\label{PropoLimiPn}
 Let {$p_\vartheta$} be a sequence such that
 \begin{equation}\label{limitpn}
   \lim_{\vartheta\to\infty}p_\vartheta = p_*,
 \end{equation}
  where each $p_\vartheta>0.$
Then 
\begin{equation} \label{limilambd}
  \lim_{\vartheta\rightarrow\infty} \lambda^{(p_\vartheta)}(G) = \lambda^{(p_*)}(G).
\end{equation}
 \end{Proposition}
\begin{proof}
We restrict the domain of $\x$ on a unit sphere, which is denoted as $\mathbb{S}^{n-1} \equiv \{\x \in \mathbb{R}^n: \x^\T \x=1\}.$
Rename the function in \eqref{sperical constrained model} as
\begin{equation*}
\hat{f}(\x,p) =(r-1)!\frac{\mathcal{A} \x^r}{\|\x\|_p^{r}} \qquad (\x, p) \in  \mathbb{S}^{n-1}\times(0,+\infty),
\end{equation*}
and we have
\begin{equation*}
 \lambda^{(p)}(G) = \max_{\x \in  \mathbb{S}^{n-1}} \hat{f}(\x,p).
\end{equation*}
Here $\hat{f}(\x,p)$ is continuous.
Let $\{\x_\vartheta^*\}$ be an infinite sequence on the compact space $\mathbb{S}^{n-1},$ such that
\begin{equation}\label{EOptipoi}
  \hat{f}(\x_\vartheta^*, p_\vartheta) = \lambda^{(p_\vartheta)}(G).
\end{equation}
If there are more than one point satisfying the equation \eqref{EOptipoi}, we randomly choose one of them to be $\x_\vartheta^*.$
Suppose $\{\x_\vartheta^*\}$ is a convergent sequence without loss of generality. Since the sequence is bounded, there exists a point $\x_0^*\in\mathbb{S}^{n-1}$ satisfying
\begin{equation}\label{Limitxn}
  \lim_{\vartheta\rightarrow \infty} \x_\vartheta^* = \x_0^*.
\end{equation}
For any  $\tilde{\x} \in \mathbb{S}^{n-1},$  we have
\begin{equation}\label{Contradic}
  \hat{f}(\tilde{\x},p_\vartheta)\leq \hat{f}(\x_\vartheta^*,p_\vartheta)
\end{equation}
from \eqref{EOptipoi},
which indicates that
\begin{equation*}
  \lim_{\vartheta\rightarrow\infty} \hat{f}(\tilde{\x},p_\vartheta)\leq \lim_{\vartheta\rightarrow\infty}  \hat{f}(\x_\vartheta^*,p_\vartheta).
\end{equation*}
Then we obtain
\begin{equation}\label{SuppLes}
   \hat{f}(\tilde{\x},p_*) \leq  \hat{f}(\x_0^*,p_*)
\end{equation}
based on \eqref{limitpn} and \eqref{Limitxn}.
Therefore we have
$
  \hat{f}(\x_0^*,p_*) = \max_{\x \in \mathbb{S}^{n-1}} \hat{f}(\x,p_*)= \lambda^{(p_*)}(G).
$ 
Since $$\hat{f}(\x_0^*,p_*) = \lim_{\vartheta \rightarrow \infty} \hat{f}(\x_\vartheta^*,p_\vartheta) =  \lim_{\vartheta \rightarrow \infty}\lambda^{(p_\vartheta)}(G),$$
conclusion \eqref{limilambd} is then obtained.

\end{proof}

\subsection{The CSRH algorithm }
We employ an iterative algorithm to solve \eqref{sperical constrained model}.
\begin{figure}
  \centering
  \includegraphics[width=.5\textwidth]{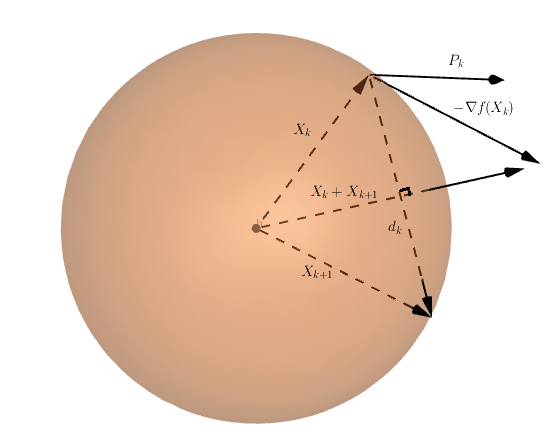}
  \caption{Illustration of the new iterate on the  unit sphere}\label{Ball}
\end{figure}
Suppose that the current iterate is a unit vector $\x_k$. Our task is to find a new iterate $\x_{k+1},$ which satisfies the following two conditions.
\begin{enumerate}
  \item  $\x_{k+1}$ is on the unit sphere;
  \item $\dd_k=\x_{k+1}-\x_k$ is an ascent direction, i.e.,
\begin{equation}\label{DirecDescPro}
  \dd_k^{ \top} \nabla f(\x_k)>0 .
\end{equation}

\end{enumerate}

In Figure \ref{Ball}, the current iterate $\x_k$ is on the unit sphere and  we can see that  $\x_{k+1}$ is a unit vector if and only if the vector
$\x_{k+1}+\x_{k}$ and the vector $\dd_k=\x_{k+1}-\x_k$ are perpendicular to each other, i.e.
\begin{equation}\label{OrthVec}
  (\x_{k+1}+\x_{k})^{\top}{\dd_{k}}=0.
\end{equation}
Let $W_k$ be a skew-symmetric matrix, i.e., $W_k=-W_k^{\top}.$ Then we have
\begin{equation*}
  (\x_{k}+\x_{k+1})^{\top} W_k(\x_{k+1}+\x_k)
  = -(\x_{k+1}+\x_{k})^{\top} W_k(\x_{k+1}+\x_k) =0.
\end{equation*}
Therefore, the equation \eqref{OrthVec} is feasible and the first condition of $\x_{k+1}$ holds when
\begin{equation}\label{SkewSyAppe}
  \dd_k = W_k(\x_{k}+\x_{k+1}).
\end{equation}
Furthermore, based on the  optimization techniques it is available to find an ascent direction  $\p_k$  such that
\begin{equation}\label{DirecDescPro2}
  \p_k^{ \top} \nabla f(\x_k)>0 .
\end{equation}
Then the existing information in Figure \ref{Ball} for us to obtain $\dd_k$ is $\p_k$ and $\x_k,$  both of which have relation with $\nabla f(\x_k)$ in \eqref{DirecDescPro2} and \eqref{xTg(x)} respectively. Hence, in order to satisfy \eqref{DirecDescPro} we construct $\dd_k$ as a combination of $\x_k$ and $\p_k,$ i.e.,
\begin{equation}\label{LinComb}
  \dd_k = a \x_k + b \p_k,
\end{equation}
and obtain
\begin{equation}\label{LinCombTgx_k}
  \dd_k^{\top} \nabla f(\x_k) = a \x_k^{\top}\nabla f(\x_k) + b \p_k^{\top}\nabla f(\x_k) =b \p_k^{\top}\nabla f(\x_k).
\end{equation}
Therefore, if $b>0$ in \eqref{LinComb}, $\dd_k$ is an ascent direction with $\dd_k^{\top} \nabla f(\x_k)>0$.

The previous analysis shows that the two conditions of $\x_{k+1}$ are valid when $\dd_k$ satisfies \eqref{SkewSyAppe} and \eqref{LinComb} for $b>0$. This motivates us to construct the skew-symmetric matrix $W_k$ by $\x_k$ and $\p_k.$ Let
\begin{equation}\label{EquaW_k}
W_k=\frac{1}{2}\alpha (\p_k \x_k^\T-\x_k \p_k^\T  )\in \mathbb{R}^{n \times n}
\end{equation}
with $\alpha$ being a positive parameter. The constant $b=\frac{1}{2}\alpha\x_k^{\top}(\x_k+\x_{k+1})$ in \eqref{LinComb}.
Since the angle between vectors $\x_k$ and $\x_k+\x_{k+1}$ is less than or equal to $\frac{\pi}{2}$ in Figure \ref{Ball}, then we have $b\geq0.$  However if $b=0,$ i.e., $\x_{k+1}=-\x_k,$ there is a contradiction when we substitute $\x_{k+1}$ by $-\x_{k}$  in \eqref{SkewSyAppe}.
Hence, we have $b>0$ and equations \eqref{SkewSyAppe} and \eqref{LinComb} hold, which means the two conditions of $\x_{k+1}$ are satisfied when $W_k$ is the matrix in \eqref{EquaW_k} with $\p_k$  being an ascent direction.
\begin{Lemma} \label{Direct computation method-Lemma}
The new iterate $\x_{k+1}$ can be expressed as
\begin{equation}\label{x_(k+1)}
\x_{k+1}(\alpha)=\frac{[(2-\alpha \x_k^\T \p_k)^2-\| \alpha \p_k  \|^2] \x_k+4\alpha \p_k}{4+\| \alpha \p_k \|^2-(\alpha \x_k^\T \p_k)^2},
\end{equation}
from \eqref{SkewSyAppe} and \eqref{EquaW_k}.
Further we have
\begin{equation}\label{Eqx_(k+1)-x_k}
\| \x_{k+1}(\alpha) -\x_k \|=2\left( \frac{\|\alpha \p_k \|^2-(\alpha \x_k^\T \p_k)^2}{4+\|\alpha \p_k \|^2-(\alpha \x_k^\T \p_k)^2}\right)^{\frac{1}{2}}.
\end{equation}
\end{Lemma}
\begin{proof}
From \eqref{SkewSyAppe}, we obtain $\x_{k+1}=Q\x_{k},$
where
\begin{equation*}
  Q = (I-W_k)^{-1}(I+W_k).
\end{equation*}
That is to say the orthogonal transform is in fact the Cayley transform. The proof is then similar to Lemma $3.2$ in  \cite{chang2016computing,chen2015computing}.
\end{proof}

For the new point $\x_{k+1}$ in \eqref{x_(k+1)}, a crucial step is to find an ascent direction $\p_k$ to guarantee the ascent property in \eqref{DirecDescPro2}. Since problems related with hypergraphs and tensors are often large and time-consuming for computation, we employ  the nonlinear conjugate gradient method, which is  proposed for large-scale nonlinear optimization problems, to acquire a suitable $\p_k$. The nonlinear conjugate gradient method  does not need the Hessian matrices of the objective function and is usually faster than the steepest descent method. In \cite{hager2005new,hager2006survey}, a  nonlinear conjugate gradient method called  CG$\_$DESCENT was given and it was proved that the  CG$\_$DESCENT possesses a good descent property.  Attracted by this merit, we adopt the construction of  parameter $\beta_k$ in CG$\_$DESCENT and obtain the ascent direction $\p_k$ by
\begin{equation}\label{Direction p_k}
\p_k=\g(\x_k)+\beta_{k-1} \dd_{k-1}.
\end{equation}
The scalar $\beta_{k-1}$ above is defined as $\beta_{k-1}=\max(0,\tilde{\beta}_{k-1})$, where
\begin{equation}\label{beta}
\tilde{\beta}_{k-1} =\left\{
\begin{aligned}
& \left(\tau\dd_{k-1} \frac{\|\textbf{y}_{k-1}\|^2}{\dd_{k-1}^\T \textbf{y}_{k-1}}-\textbf{y}_{k-1}\right)^\T \frac{\g{(\x_k)}}{\dd_{k-1}^\T \textbf{y}_{k-1}} \quad \text{if} \,| \dd_{k-1}^\T \textbf{y}_{k-1}  |\geq  \epsilon \|\dd_{k-1}\| \|\y_{k-1}\| \\
& \qquad \qquad \quad \qquad  0 \qquad  \qquad  \qquad \qquad \qquad  \quad \quad\text{otherwise,}  \\
\end{aligned}
\right.
\end{equation}
$\textbf{y}_{k-1}=\g(\x_k)-\g(\x_{k-1}),$  parameters $\frac{1}{4}<\tau <1 $ and $\epsilon>0.$
The initial direction is chosen as
$\p_0=\g(\x_0).$
The direction $\p_k$ in \eqref{Direction p_k} is proved to satisfy the ascent property in the following Lemma.
\begin{Lemma}\label{Propertiyofp_k}
  The search direction $\p_k$ generated by \eqref{Direction p_k} satisfies the sufficient ascent condition, i.e.
\begin{equation}\label{ascent property}
\p_k^\T \g(\x_k)\geq \left(1-\frac{1}{4\tau} \right) \|\g(\x_k)\|^2,
\end{equation}
and there exists a constant $M_0>1$ such that
 \begin{equation}\label{bounded p_k}
\|\p_k\| \leq M_0 \|\g(\x_k)\|.
 \end{equation}
\end{Lemma}
\begin{proof}
  When $\beta_k=0,$ it is easy to show that the two inequalities hold.
For $\beta_k \neq 0,$ we have
  \begin{equation*}
    \begin{aligned}
    \p_k^\T \g(\x_k) &= \|\g(\x_k)\|^2+ \tau \frac{\dd_{k-1}^\T\g(\x_k)}{\dd_{k-1}^\T\y_{k-1}} \frac{\|\y_{k-1}\|^2}{\dd_{k-1}^\T\y_{k-1}}\dd_{k-1}^\T\g(\x_k)
    - \frac{\y_{k-1}^\T\g(\x_k)}{\dd_{k-1}^\T\y_{k-1}} \dd_{k-1}^\T\g(\x_k) \\
    & = \frac{1}{4 \tau}\|\g(\x_k)\|^2 -\frac{\dd_{k-1}^\T\g(\x_{k})}{\dd_{k-1}^\T\y_{k-1}}\y_{k-1}^\T\g(\x_{k}) +\tau \frac{(\dd_{k-1}^\T\g(\x_{k}))^2}{(\dd_{k-1}^\T\y_{k-1})^2} \|\y_{k-1}\|^2 \\  &\quad +\left(1-\frac{1}{4\tau} \right) \|\g(\x_k)\|^2 \\
    & \geq \left(1-\frac{1}{4\tau} \right) \|\g(\x_k)\|^2.
    \end{aligned}
  \end{equation*}
 Since $$\|\dd_{k-1}\cdot\textbf{y}_{k-1}^\T \|=\| \dd_{k-1}\| \cdot \|\textbf{y}_{k-1}\| \quad \text{and} \quad \|\dd_{k-1}\cdot\dd_{k-1}^\T  \|=\|\dd_{k-1}\|^2 ,$$  we obtain
\begin{equation*}
 \begin{aligned}
\|\beta_{k-1} \dd_{k-1}\|&\leq \left\| \frac{\tau \| \textbf{y}_{k-1}\|^2 \cdot \dd_{k-1}\cdot \dd_{k-1}^\T -\dd_{k-1}^\T \textbf{y}_{k-1} \cdot \dd_{k-1}\cdot\textbf{y}_{k-1}^\T  }
{(\dd_{k-1}^\T \textbf{y}_{k-1})^2} \right\| \cdot \| \g(\x_k) \| \\
& \leq \left[  \frac{ \|\dd_{k-1}\| \|\textbf{y}_{k-1} \|}{|\dd_{k-1}^\T \textbf{y}_{k-1}|}
+  \frac{\tau  \| \textbf{y}_{k-1}\|^2  \|\dd_{k-1}\|^2 }{(\dd_{k-1}^\T \textbf{y}_{k-1})^2}   \right]\cdot \| \g(\x_k) \| \\
&\leq \left[\frac{1}{\epsilon} +  \frac{\tau}{\epsilon^2} \right] \|\g(\x_k)\|.\\
 \end{aligned}
 \end{equation*}
  Then we deduce that
  \begin{equation*}
 \begin{aligned}
\|\p_k\| \leq \|\g(\x_k)\|+\|\beta_{k-1} \dd_{k-1}\| \leq \left[1+\frac{1}{\epsilon}+ \frac{\tau }{\epsilon^2}\right]\|\g(\x_k)\|.
 \end{aligned}
 \end{equation*}
 Inequality \eqref{bounded p_k} is valid when $M_0=1+\frac{1}{\epsilon}+ \frac{\tau}{\epsilon^2}.$
\end{proof}

In the curvilinear line search,  the parameter $\alpha$ in \eqref{x_(k+1)} is determined to ensure that the Wolfe conditions hold. We provide the details  in the next subsection.
\subsection{Feasibility of Wolfe conditions}
In this section we prove that there exists a step length $\alpha_k$ satisfying the Wolfe conditions for the curvilinear search in \eqref{x_(k+1)} in each iteration. First, we compute the derivative of $\alpha$ which plays an important role in line search.
\begin{Lemma}
Let $f'(\alpha)$ be the derivative of $f(\x_{k+1}(\alpha))$ at point $\alpha.$ Then we have
\begin{equation}\label{Alp_f'}
\alpha f'(\alpha)=-\g(\x_{k+1}(\alpha))^\T \x_k.
\end{equation}
\end{Lemma}
\begin{proof}
Equation \eqref{x_(k+1)} means that
 $$[4+\alpha^2\|\p_k\|^2-\alpha^2(\x_k^\T\p_k)^2]\x_{k+1}(\alpha)=[(2-\alpha \x_k^\T \p_k)^2-\alpha^2 \|\p_k\|^2]\x_k+4 \alpha \p_k.$$
 Then we take derivative with respect to  $\alpha$ as follows
\begin{equation} \label{Der-Alp}
\begin{aligned}
   &2\alpha(\|\p_k\|^2-(\x_k^\T\p_k)^2)\x_{k+1}(\alpha)+[4+\alpha^2\|\p_k\|^2-\alpha^2(\x_k^\T\p_k)^2]\x'_{k+1}(\alpha)  \\
   =&[-4\x_k^\T\p_k+2\alpha (\x_k^\T\p_k)^2- 2\alpha\|\p_k\|^2]\x_k+4\p_k.
\end{aligned}
\end{equation}
By multiplying both sides of \eqref{Der-Alp} by $\alpha$ we get
\begin{equation} \label{Alp_x'(k+1)}
  \alpha \x'_{k+1}(\alpha)=\frac{-2\alpha^2(\|\p_k\|^2-(\x_k^\T\p_k)^2)}{4+\alpha^2\|\p_k\|^2-\alpha^2(\x_k^\T\p_k)^2}\x_{k+1}(\alpha)+\x_{k+1}(\alpha)-\x_k
\end{equation}
from \eqref{x_(k+1)}.
Since $\g(\x_{k+1}(\alpha))^\T\x_{k+1}(\alpha)=0,$  from \eqref{Alp_x'(k+1)} we obtain
\begin{equation*}
  \alpha f'(\alpha)=\alpha \g(\x_{k+1}(\alpha))^\T\x'_{k+1}(\alpha)=- \g(\x_{k+1}(\alpha))^\T\x_k.
\end{equation*}
\end{proof}
Since $f(\x)$ is twice continuously differentiable in the compact set $\mathbb{S}^{n-1}$ , we can find a constant $M$ such that
\begin{equation} \label{boundary}
|f(\x)|\leq M, \quad \|\g(\x)\| \leq M, \quad \text{and} \quad \|\nabla^2f(\x)\| \leq M.
\end{equation}

For a given optimization algorithm which enjoys a good ascent or descent property, it is proved that  step lengths that satisfy the  Wolfe conditions exist for a monotonous line search  in \cite[Lemma 3.1]{nocedal2006numerical}. In the following theorem we prove that Wolfe conditions are practicable for the curvilinear line search in our algorithm.
\begin{Theorem}\label{wolfe condition theorem}
If \ $0<c_1<c_2<1 $,  there exists $\alpha_k>0$ satisfying
\begin{eqnarray}
f\big(\x_{k+1}(\alpha_k)\big)&\geq &f(\x_k)+c_1 \alpha_k \g(\x_k)^\T \p_k,   \label{wolfe-1}\\
\g(\x(\alpha_k))^\T\p_k & \leq &c_2 \g(\x_k)^\T\p_k. \label{wolfe-2}
\end{eqnarray}
\end{Theorem}
\begin{proof}
Let $\x(\alpha)=\x_{k+1}(\alpha)$ and $f(\alpha)=f(\x_{k+1}(\alpha)).$
From \eqref{x_(k+1)}, we have $\x_{k+1}'(0) = - \x_k^\T \p_k \x_k+\p_k,$ and
\begin{equation*}
\begin{aligned}
 f'(0)&=\left.\frac{\mathrm{d} f(\x_{k+1}(\alpha))}{\mathrm{d} \alpha}\right|_{\alpha=0} = \g(\x_{k+1}(0))^\T \x_{k+1}'(0)  \\
&=\g(\x_k)^\T (- \x_k^\T \p_k \x_k+\p_k)  = \g(\x_k)^\T \p_k.
\end{aligned}
\end{equation*}
Denote a linear function $l(\alpha)=f(\x_k)+c_1\alpha \g(\x_k)^\T \p_k.$  Then $f(0)=l(0)=f(\x_k)$ and $f'(0) > l'(0)>0$ due to $0<c_1<1$ and $\g(\x_k)^\T\p_k>0 $ in \eqref{ascent property}.
Since $f(\alpha)$ is bounded above, the graph of $f(\alpha)$ must intersect with the line $l(\alpha)$ at least once when $\alpha>0$. Suppose $\bar{\alpha}$ is the smallest intersection point, we obtain
\begin{equation}\label{intersection}
f(\x_{k+1}(\bar{\alpha}))= f(\x_{k})+c_1 \bar{\alpha} \g(\x_k)^\T \p_k.
\end{equation}
By the mean value theorem, we can find  $\rho \in (0,\bar{\alpha})$ satisfying
\begin{equation}\label{meanvalue-theorem}
\begin{aligned}
f(\x_{k+1}(\bar{\alpha}))-f(\x_k)&= \bar\alpha f'(\rho)\\
                     [\text{By} \; \eqref{Alp_f'}]     \qquad   & =-\frac{\bar{\alpha}}{\rho}\g(\x_{k+1}(\rho))^\T\x_k.
\end{aligned}
\end{equation}
On the other hand, from \eqref{xTg(x)} and \eqref{x_(k+1)} we have
\begin{equation*}
\begin{aligned}
\g(\x_{k+1}(\rho))^\T\x_{k+1}(\rho) &=\frac{[(2-\rho \x_k^\T\p_k)^2-\|\rho \p_k\|^2]\g(\x_{k+1})(\rho)^\T\x_k}{4+\|\rho \p_k\|^2-(\rho \x_k^\T\p_k)^2} \\
& \quad +\frac{4\rho \g(\x_{k+1}(\rho))^\T\p_k }{4+\|\rho \p_k\|^2-(\rho \x_k^\T\p_k)^2}\\
&=0.
\end{aligned}
\end{equation*}
Then we have
\begin{equation}\label{d_k p_k}
-[(2-\rho \x_k^\T\p_k)^2-\|\rho \p_k\|^2]\g(\x_{k+1}(\rho))^\T\x_k=4\rho \g(\x_{k+1}(\rho))^\T\p_k
\end{equation}
Combining
\eqref{meanvalue-theorem} and \eqref{d_k p_k}, we have
\begin{equation}\label{1}
 [(2-\rho \x_k^\T\p_k)^2-\|\rho \p_k\|^2][f(\x_{k+1}(\bar{\alpha}))-f(\x_k)]= 4 \bar\alpha \g(\x_{k+1}(\rho))^\T\p_k.
\end{equation}
Further, from \eqref{intersection} we obtain
\begin{equation}\label{2}
\begin{aligned}
&[(2-\rho \x_k^\T\p_k)^2-\|\rho \p_k\|^2][f(\x_{k+1}(\bar{\alpha}))-f(\x_k)] \\
=&[(2-\rho \x_k^\T\p_k)^2-\|\rho \p_k\|^2]c_1 \bar{\alpha}\g(\x_k)^\T\p_k.
\end{aligned}
\end{equation}
Combing \eqref{1} and \eqref{2} we have
\begin{equation*}
  4 \g(\x_{k+1}(\rho))^\T\p_k =[(2-\rho \x_k^\T\p_k)^2-\|\rho \p_k\|^2]c_1 \g(\x_k)^\T\p_k
\end{equation*}
Since \begin{equation*}
\begin{aligned}
\x_k^\T\p_k&=\x_k^\T(\g(\x_k)+\beta_{k-1}\dd_{k-1})\\
&=\beta_{k-1}\x_k^\T(\x_k-\x_{k-1})\\
&=\beta_{k-1}(1-\x_k^\T\x_{k-1})\\
&\geq 0
\end{aligned}
\end{equation*}
and $|\x_k^\T\p_k| \leq \|\p_k\|$, we have
\begin{equation*}
\begin{aligned}
(2-\rho \x_k^\T\p_k)^2-\|\rho \p_k\|^2&=4-4\rho \x_k^\T\p_k+(\rho \x_k^\T\p_k)^2-\|\rho \p_k\|^2 \\
&\leq 4-4\rho \x_k^\T\p_k\\
& \leq 4.
\end{aligned}
\end{equation*}
Since $\g(\x_k)^\T\p_k \geq 0$,
\begin{equation}
\g(\x_{k+1}(\rho))^\T\p_k \leq c_1 \g(\x_k)^\T\p_k.
\end{equation}
Since $c_2>c_1$,  inequality  \eqref{wolfe-2} holds when $\alpha_k=\rho.$ Also from the condition $\rho \in (0, \bar{\alpha})$, we have $f(\alpha_k)>l(\alpha_k)$ and \eqref{wolfe-1} is obtained.
\end{proof}

\renewcommand{\thealgorithm}{CSRH }
\begin{algorithm}[tb!]
\caption{Computing $p$-spectral radius of a hypergraph}\label{alg}
\label{Algorithm}
\begin{algorithmic}[1]
  \STATE
For a uniform hypergraph $G$,  $p>0$,  choose parameters $ 0<c_1<c_2<1 $, $\frac{1}{4}<\tau<1$ , $\epsilon>0,$   an initial unit point $\x_0, $ and $k \gets 0$. Calculate $\p_0=\g(\x_0).$
 \WHILE{the sequence of iterates does not converge}
    \STATE  Use interpolation method to find $\alpha_k$
       such that \eqref{wolfe-1} and\eqref{wolfe-2} hold.

    \STATE Update the new iterate $\x_{k+1}=\x_{k+1}(\alpha_k)$ by \eqref{x_(k+1)}.
    \STATE Compute $\dd_k$, $\g(\x_{k+1})$ , $\beta_k,$  and $\p_{k+1}$  by \eqref{Direction p_k} .
   \STATE $k \gets k+1.$
     \ENDWHILE
\end{algorithmic}
\end{algorithm}
Up to now, the algorithm CSRH for computing the $p$-spectral radius of a hypergraph is available. First we transform the original model of $\lambda^{(p)}(G)$ into an equivalent constrained optimization problem on the unit sphere \eqref{sperical constrained model}. To solve the constrained model, we compute the ascent direction $\p_k$ from \eqref{gradient}, \eqref{beta} and \eqref{Direction p_k}, and choose a proper $\alpha_k$ so that the next iterate gained via \eqref{x_(k+1)} satisfies the Wolfe conditions \eqref{wolfe-1} and \eqref{wolfe-2}.
A fast computation method for calculating  $\mathcal{A} \x^r$ and $\mathcal{A} \x^{r-1}$ was proposed  in \cite{chang2016computing}, which improves the efficiency of products of adjacency tensor and vector.  We also adopt this technique in our algorithm.

\section{Convergence analysis}
In this section  we  prove that the CSRH algorithm converges to a stationary point of $f(\x)$ and touches the exact $p$-spectral radius with a high probability. Our CSRH algorithm terminates  finitely when there exits a constant $c$ such that $\g(\x_c)=0.$ The following  convergence  analysis is for the case that the sequence $\{\x_k\}$ is infinite and $\g(\x_k)$ is always a nonzero vector.

\subsection{ Convergence results}
Next theorem shows that CSRH algorithm is  convergent.
\begin{Theorem}\label{GradConv}
Suppose the sequence $\{\x_k\}$ is generated by the algorithm CSRH from any $\x_0 \in \mathbb{S}^n$. Then we have
$$\lim_{k\rightarrow \infty} \|\g(\x_k)\|=0.$$
\end{Theorem}
\begin{proof}
The  demonstration is divided into two steps. First, we show that the Zoutendijk condition holds, i.e.,
\begin{equation}\label{Zoutendijk}
\sum_{k=0}^{\infty} \cos^2 \varphi_k\|\g(\x_k)\|^2 < \infty.
\end{equation}
Here $\varphi_k$ is the angle between $\g(\x_k)$ and $\p_k$, which is denoted as $$\varphi_k \equiv \arccos \frac{\g(\x_k)^\T\p_k}{\|\g(\x_k)\|\|\p_k\|}.$$
Since $\nabla^2 f(\x)$ is bounded, we have $\g(\x)$ is Lipschitz continuous on $\mathbb{S}^{n-1}$, i.e.,
\begin{equation} \label{Lipschitz}
\|\g(\x_1)-\g(\x_2)\| \leq L\|\x_1-\x_2\|  \qquad  \forall \, \x_1,\x_2 \in \mathbb{S}^n
\end{equation}
for a constant $L>0$.
From \eqref{EquaW_k}, we have
 $$\|W\|=\| \frac{\alpha_k}{2}(\x_k\p_k^\T-\p_k\x_k^\T)\| \leq \frac{\alpha_k}{2}(\|\x_k\p_k^\T\|+\|\p_k\x_k^\T\|)\leq \alpha_k\|\p_k\|.$$
Hence from \eqref{SkewSyAppe}
\begin{equation}\label{Zou_2}
\|\x_{k+1}-\x_k\|\leq\|W_k\|  (\|\x_{k+1}\|+\|\x_k\|) \leq 2\alpha_k \|\p_k\|.
\end{equation}
From \eqref{Lipschitz} and \eqref{Zou_2}, we have
\begin{equation*}
(\g(\x_k)-\g(\x_{k+1}))^\T\p_k \leq L\|\x_{k+1}-\x_k\|\|\p_k\| \leq 2L \alpha_k \|\p_k\|^2.
\end{equation*}
From \eqref{wolfe-2}, we obtain
\begin{equation}\label{Zou_1}
(\g(\x_{k+1})-\g(\x_k))^\T\p_k \leq (c_2-1)\g(\x_k)^\T\p_k.
\end{equation}
By using the above two relations, we can derive the inequality
\begin{equation*}
(1-c_2)\g(\x_k)^\T\p_k \leq 2L \alpha_k \|\p_k\|^2,
\end{equation*}
which implies
\begin{equation}\label{alphageq}
 \alpha_k \geq \frac{1-c_2}{2L}\frac{\g(\x_k)^\T\p_k}{\|\p_k\|^2}.
\end{equation}
Then from \eqref{wolfe-1}, we obtain
\begin{equation*}
  f(\x_{k+1})-f(\x_k)\geq \frac{c_1(1-c_2)}{2L} \frac{(\g(\x_k)^\T\p_k)^2}{\|\p_k\|^2}= \frac{c_1(1-c_2)}{2L} \cos^2\varphi_k \|\nabla f(\x_k)\|^2,
\end{equation*}
which derives the following inequality
\begin{equation*}
  f(\x_{k+1})-f(\x_0)=\sum_{i=0}^{k} f(\x_{i+1})-f(\x_{i})\geq \frac{c_1(1-c_2)}{2L} \sum_{i=0}^{k} \cos^2\varphi_i \|\nabla f(\x_i)\|^2.
\end{equation*}
Since $f(\x)$ is bounded in \eqref{boundary}, the inequality \eqref{Zoutendijk} is then deduced.

Next, we show that the angle $\varphi_k$ is bounded away from $\frac{\pi}{2}$. By combining \eqref{ascent property} and \eqref{bounded p_k}, we obtain
\begin{equation} \label{UbounofAngle}
\begin{aligned}
\frac{\g(\x_k)^\T\p_k}{\|\g(\x_k)\|\|\p_k\|}\geq(1-\frac{1}{4\tau})\frac{\|\g(\x_k)\|}{\|\p_k\|}\geq\frac{1}{M_0}(1-\frac{1}{4\tau})\equiv C_0.
\end{aligned}
\end{equation}
The above inequalities indicate that $$\cos \varphi_k \geq C_0>0.$$
Therefore, from \eqref{Zoutendijk} we have
\begin{equation*}
\lim_{k\rightarrow \infty} \|\g(\x_k)\|=0.
\end{equation*}

\end{proof}
Recall that the graph of a function $h(\x)$ is defined as $$\text{Gr} \, h:=\{(\x,\lambda)\in \R^n \times \R: f(\x)=\lambda\}. $$
For the function $f(\x)$ involved in our problem  \eqref{sperical constrained model}, we have
$$\text{Gr} \, f=\{(\x,\lambda): [(r-1)!\Adj \x^r]^p=\lambda^p(\sum_{i}|x_i|^p)^r\},$$
where $p$ and $r$ are positive integers.
 Since $\text{Gr} \, f$ is a semialgebraic set, $f(\x)$ is a semialgebraic function and satisfies the  {\L}ojasiewicz inequality \cite{absil2005convergence,bolte2007lojasiewicz,yang2013ablock}, which means that for a critical point $\x_*$ of $f(\x)$, there exist constants $\theta \in [0,1)$ and $C_1>0$, as well as $\mathscr{U}$ being a neighbourhood of $\x_*$  such that
\begin{equation}\label{KL property}
  |f(\x)-f(\x_*)|^{\theta}\leq C_1\|\g (\x)\|
\end{equation}
for $\x \in \mathscr{U}.$
The next theorem shows that if the sequence $\{\x_k\}$ is infinite, it has a unique accumulation point.
\begin{Theorem}
  Assume the infinite sequence $\{\x_k\}$ is generated by the CSRH algorithm. Then it converges to a unique point $\x_*,$ that is,
  \begin{equation*}
    \lim_{k\rightarrow \infty}\x_k=\x_*,
  \end{equation*}
  and $\x_*$ is a first-order stationary point.
\end{Theorem}
\begin{proof}
From  \eqref{alphageq}, \eqref{ascent property} and  \eqref{bounded p_k} we have
\begin{eqnarray}
  \alpha_k &\geq &\frac{1-c_2}{2L}(1-\frac{1}{4\tau})\frac{\|\g(\x_k)\|^2}{\|\p_k\|^2} \nonumber\\
            &\geq& \frac{1-c_2}{2L M_0^2}(1-\frac{1}{4\tau})\nonumber \\
            & \equiv& \alpha_{\min} > 0. \nonumber
\end{eqnarray}
Moreover, from \eqref{wolfe-1} and \eqref{ascent property}  we obtain
\begin{eqnarray}
  f(\x_{k+1})-f(\x_k)  &\geq& c_1 \alpha_k \g(\x_k)^\T\p_k \nonumber\\
&\geq& c_1 \alpha_{\min} (1-\frac{1}{4\tau}) \|\g(\x_k)\|^2 \nonumber\\
   &>& 0. \label{FStricIncr}
\end{eqnarray}
 We take no account of condition $\|\g(\x_k)\|= 0$ under which the algorithm terminates finitely.
 The above inequality indicates that
 \begin{equation}\label{Lcond1}
  [f(\x_{k+1})=f(\x_k) ]  \Rightarrow [ \x_{k+1}=\x_k].
 \end{equation}
Based on \eqref{bounded p_k}, \eqref{wolfe-1} and \eqref{Zou_2}, we have
\begin{eqnarray}
 f(\x_{k+1})-f(\x_k)&\geq&  c_1 \alpha_{k} (1-\frac{1}{4\tau})  \frac{\|\g(\x_k)\|\|\p_k\|}{M_0} \nonumber\\
          &\geq&  (1-\frac{1}{4\tau})  \frac{ c_1 }{2M_0} \|\g(\x_k)\| \|\x_{k+1}-\x_k\| \label{Lcond2}
\end{eqnarray}
From \eqref{Lcond1} and \eqref{Lcond2}, as well as the {\L}ojasiewicz inequality \eqref{KL property}, we have the conclusions hold based on \cite[Theorem 3.2]{absil2005convergence}.
\end{proof}

\subsection{Probability of obtaining the exact $p$-spectral radius}
 Due to the feasibility of {\L}ojasiewicz inequality in \eqref{KL property}, we get the probability of the CSRH method touching the true $p$-spectral radius.
\begin{Proposition}[ ]\label{ProTheo}
  Suppose CSRH algorithm is implemented from $N$ uniformly distributed initial points on $\mathbb{S}^{n-1}$ for $N$ times. We take the largest one among the results of these trails as the $p$-spectral radius of the relevant problem. The probability of getting the exact $p$-spectral radius is
$$1-(1-\zeta)^N,$$
in which $\zeta$ is a constant satisfying $\zeta\in(0,1].$
 If $N$ is large enough, the probability is high.
\end{Proposition}
\begin{proof}
 This Proposition can be proved in the way  similar to \cite[Theorem 4.9]{chang2016computing}. We omit the details.
\end{proof}

\section{Numerical experiments}
In this section, we show the performance of CSRH for computing $p$-spectral radii of both small and large scale hypergraphs. We  compare our method with several existing methods for computing eigenvalues of adjacency tensors. Examples of approximating the Lagrangian of a hypergraph are given in Subsection 2.  All experiments  are carried out by using MATLAB version R2015b and Tensor Toolbox version 2.6 \cite{TTB_Software}. The experiments in Subsections 5.1 and 5.2 are terminated when
 $$\|\g(\x)\|\leq 10^{-8}\quad \text{or}\quad \|\lambda^{(p)}-\lambda^{(p)}_*(G)\|\leq 10^{-12},$$
 where $\lambda^{(p)}$ is our computed $p$-spectral radius and $\lambda^{(p)}_*(G)$ is the exact result obtained from  theorems or conclusions in existing  literature.
The maximum iteration of CSRH is taken as 1000 for all algorithms except those performed by the MATLAB function in Tensor Toolbox. For each experiment in this section,  we compute 100 times to obtain  100 estimated values  $\lambda^{(p)}_1,\ldots, \lambda^{(p)}_{100}$ and choose the largest one as our computational result of the $p$-spectral radius related with $G$. When $\lambda^{(p)}_*(G)$ is attainable, the accuracy rate of the CSRH algorithm is defined as
\begin{equation}
  \text{Accu.} \equiv \bigg| \bigg\{ i: \frac{|\lambda^{(p)}_i-\lambda^{(p)}_*(G)|}{|\lambda^{(p)}_*(G)|} \leq 10^{-8} \bigg\}\bigg| \times 1\%.
\end{equation}
Each number of  iterations (Iter.) and computational time (Time) we reported in this section is the sum of corresponding quantities for all 100 executions of the experiment. The relative errors (Err.)  between the numerical results and the exact solutions are provided.

\subsection{Computation of $p$-spectral radii of hypergraphs}
We compare the following three algorithms for computing eigenvalues of adjacency tensors associated with different hypergraphs:
\begin{itemize}
  \item An adaptive shifted power method \cite{kolda2011shifted} SS-HOPM. This method can be invoked by \texttt{eig\_sshopm} in Tensor Toolbox 2.6 for Z-eigenvalues of symmetric tensors.
 \item A first-order optimization algorithm CEST \cite{chang2016computing} which is proposed for eigenvalues of  large scale sparse tensors involving  even order hypergraphs.
  \item CSRH: the method proposed in Section 3.
\end{itemize}

\textbf{Example 1 ($\mathbf{p=2}$).}
First, we compute the largest Z-eigenvalues of adjacency tensors of the following hypergraphs:
\[
\left\{
  \begin{array}{ll}
    G_1: & \hbox{V=\{1,2,3,4\}\quad\text{and}\quad E=\{123,234\};} \\
  G_2: & \hbox{V=\{1,2,3,4,5,6,7\}\quad\text{and}\quad E=\{123,345,567\};} \\
    G_3: & \hbox{V=\{1,2,3,4,5\}\quad\text{and}\quad E=\{123,345\};} \\
    G_4: & \hbox{V=\{1,2,3,4\}\quad\text{and}\quad E=\{123,124,134,234\}.}
  \end{array}
\right.
\]
The first hypergraph $G_1$ is given in \cite{xie2013z} as Example 1, while the last three hypergraphs are Example $4,$ $7$ and $9$ in \cite{pearson2014spectral}.
The hypergraph $G_4$ is  actually a tetrahedron. 

In Table \ref{ZEigG}, we demonstrate results of CSRH and  SS-HOPM for computing the largest Z-eigenvalues of adjacency tensors of some small hypergraphs.
Since all the four hypergraphs given above are of odd orders, the comparison does not include CEST method, which is designed for even order hypergraphs.
The Err. column  shows the relative error between the computational result and the exact largest  Z-eigenvalue  provided in the corresponding references.
Under the condition that the relative error reaches $10^{-16}$, our CSRH method is much more stable and efficient than the SS-HOPM method.
\begin{table}[ht]
\footnotesize
\centering
\begin{tabulary}{1.0\textwidth}{c|llll|llll} 
\hline
\multirow{2}{5em}{Hypergraph}& \multicolumn{4}{c|}{CSRH}   &\multicolumn{3}{c}{SS-HOPM}\\\cline{2-9}
                      & Iter.& Time(s)&Accu.  &Err.& Iter.& Time(s) &Accu. & Err. \\
\hline
  $G_1$& 13593 &3.35&1.00  & $5.44\times10^{-16}$&2668 & 4.89& 1.00& $5.44\times10^{-16}$\\
   $G_2$ & 1257& 0.78&1.00  &$3.85\times10^{-16}$ &18610 &32.58&0.94  &$3.85\times10^{-16}$\\
  $G_3$ & 674 & 0.42 &1.00 &$3.85\times10^{-16}$ &731&1.61&1.00   &$7.69\times10^{-16}$ \\
   $G_4$ & 8901 & 2.23&0.18 & $1.48\times10^{-16}$&2317 &4.38& 0.22  & $2.96\times10^{-16}$\\\hline

  \hline
\end{tabulary}
\caption{Z-Eigenvalues of  adjacency tensors of several small hypergraphs. }
\label{ZEigG}
\end{table}

In the next experiment, we study the  probability of CSRH method  getting the true largest Z-eigenvalue of $G_4$ and show that the probability increases along with the trail times.  We employ the CSRH method to compute the largest Z-eigenvalue of the adjacency tensor of $G_4$  from uniformly distributed and randomly chosen initial points. Once the relative error between the computational largest Z-eigenvalue and its exact value $3/2$ reaches $10^{-8},$ the experiment is terminated and we  record the number of trails.  This experiment is repeated for one thousand times.  Let $\sigma(i)$  be the total occurrence  of experiments whose trail time is the integer $i.$  The   frequency of touching the exact Z-eigenvalue when running $i$ times is
\begin{equation}
\nu_i= \frac{\sum_{j\leq i}\sigma(j)}{1000}.
\end{equation}
\begin{figure}[h]
\centering
  \includegraphics[width=.5\textwidth]{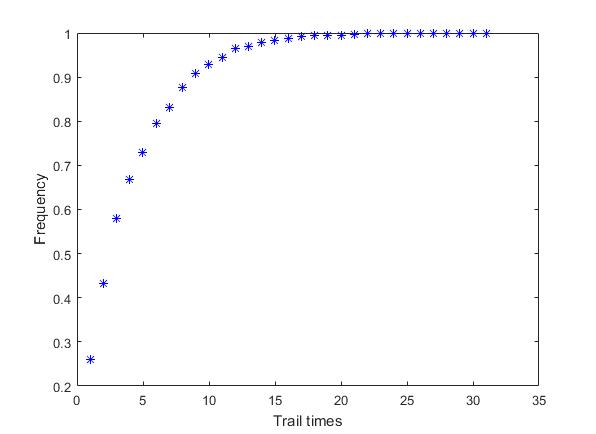}
\caption{Probability of touching the exact largest Z-eigenvalue of adjacency of $G_4$.}\label{Frequency}
\end{figure}
In Figure \ref{Frequency}, we display the relation between trail times and success probability. It illustrates that the probability tends to one along with the increase of trail times $i,$ which coincides with the conclusion in Theorem \ref{ProTheo}.

\textbf{Example 2 ($\mathbf{p=r}$).}
Next, we compare  CEST and  CSRH methods for computing the largest H-eigenvalues of  adjacency tensors of loose paths. An $r$-graph with $m$ edges is called a loose path if its vertex set is
$$V=\big\{i_{(1,1)},\ldots,i_{(1,r)},i_{(2,2)},\ldots,i_{(2,r)},\ldots,i_{(m,2)},\dots,i_{(m,r)}\big\}$$
and its edge set is
$$E=\big\{ \{i_{(1,1)},\ldots,i_{(1,r)}\},\{i_{(1,r)},i_{(2,2)},\ldots,i_{(2,r)}\},\ldots,\{i_{(m-1,r)},i_{(m,2)},\dots,i_{(m,r)}\}\big\}.$$
An $r$-uniform loose path with $m$ edges has $m(r-1)+1$ vertices. For example, the $6$-unform loose path with $4$ edges in Figure \ref{Loosepath}  has $21$ vertices.
\begin{figure}[h]
\centering

  \includegraphics[width=.6\textwidth]{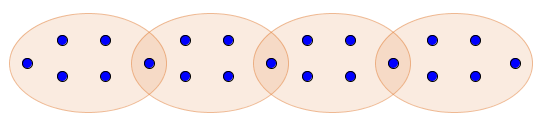}

\caption{A 6-uniform loose path with 4 edges. }\label{Loosepath}
\end{figure}
The following theorem proved in \cite{yue2016largest} offers a convenient way to acquire the largest H-eigenvalues of adjacency tensors of loose paths with $m=3$ or $m=4$.

\begin{Theorem}[\cite{yue2016largest}]\label{looPatTheor}
Let $G$ be an $r$-uniform loose path with $m$ edges and $\lambda_H(G)$ be the largest H-eigenvalue of its adjacency tensor $\mathcal{A}$. Then we have
\begin{enumerate}
  \item $\lambda_H(G)=\big(\frac{1+\sqrt{5}}{2}\big)^{\frac{2}{r}}$  for $m=3,$
  \item $\lambda_H(G)=3^{\frac{1}{r}}$ for $m=4.$
\end{enumerate}
\end{Theorem}

In Table \ref{HEigLooPa}, we compare CSRH and CEST for computing the largest H-eigenvalues of adjacency tensors of different loose paths.
The column Err. presents the relative error between our computed result  and the exact one given by Theorem \ref{looPatTheor}.
When relative error achieves precision of $10^{-16},$  the CSRH method
 saves at least $75\%$ of the time CEST takes  in every problem.  The comparison between CEST and CSRH  verifies that the high efficiency of CSRH method does not only relies on the fast computation technique in \cite{chang2016computing}, because CEST method use this technique as well.
\begin{table}[ht]
\footnotesize
\centering
\begin{tabulary}{1.0\textwidth}{c|l|llll|llll}
\hline
\multirow{2}{2em}{$m$ }&  \multirow{2}{3em}{$r$ }& \multicolumn{4}{c|}{CSRH}& \multicolumn{4}{c}{CEST} \\\cline{3-10}
    &&Iter. & Time(s)& Accu.& Err.&Iter. & Time(s)& Accu. & Err. \\\hline
\multirow{4}{2em}{3} &4&38123&9.14& 1.00&$3.49\times10^{-16}$ & 42760&70.28&1.00&$3.49 \times10^{-16}$ \\
                     &6&62780&17.55& 0.97&$5.67\times10^{-16}$ & 65706&105.53&0.99 &$7.56\times10^{-16}$\\ 
                      &8&71311&23.38&0.66 & $3.94\times10^{-16}$ & 76778&106.95& 0.65&$7.88\times10^{-16}$\\\hline

\multirow{4}{2em}{4} &4&69517&16.92&1.00 &$5.06\times10^{-16}$& 49331&79.81& 1.00&$5.06\times10^{-16}$\\
                    &6&86171&24.83& 0.96&$5.55\times10^{-16}$& 76105&113.11& 0.98& $5.55\times10^{-16}$ \\
                      &8&75907&24.71&0.33 &$7.74\times10^{-16}$& 91690&106.57& 0.42 &$9.68\times10^{-16}$\\\hline
\end{tabulary}
\caption{H-eigenvalues of adjacency tensors of loose paths. }
\label{HEigLooPa}
\end{table}

\textbf{Example 3.}
If all edges of a hypergraph share a same vertex, then it is called a $\beta$-star. An $r$-uniform $\beta$-star with $m$ edges have $m(r-1)+1$ vertices.
\begin{figure}[h]
\centering

  \includegraphics[width=.35\textwidth]{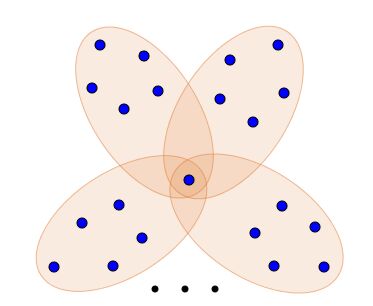}

\caption{A $6$-uniform $\beta$-star.  }\label{Betastar}
\end{figure}
We present a class of  $6$-uniform $\beta$-star  in Figure \ref{Betastar} as an example.
\begin{table}
\centering
\begin{tabular}{cccc}
\scriptsize
 \begin{tabulary}{1.0\textwidth}{l|llll} 
 \hline
\multirow{2}{2em}{n}&\multicolumn{4}{c}{ $p=3,\, r=3 \, \, (p> r-1)$}    \\\cline{2-5}
                     &Iter. & Time(s) &Accu. & Err. \\\hline
 21  & 1835& 0.34&1.00 &$5.38\times10^{-16}$\\
201  & 2609& 0.60&1.00&$3.55\times10^{-15}$ \\
2,001  &3539&1.87&1.00 &$4.33\times10^{-14}$ \\
20,001  & 4475& 12.93&1.00 &$6.39\times10^{-14}$\\
200,001  & 6038&263.39&0.98&$1.93\times10^{-11}$ \\
2,000,001  & 20018& 15437.99&1.00&$1.22\times10^{-10}$ \\
\hline
\multicolumn{5}{c}{The $3$-spectral radius of $3$-uniform $\beta$-stars ( $p>r-1$ )}\\
\end{tabulary}
 &
\scriptsize
\begin{tabulary}{1.0\textwidth}{l|llll} 

 \hline
\multirow{2}{2em}{n}&\multicolumn{4}{c}{ $p=4,\, r=6 \, \, (p< r-1)$}    \\\cline{2-5}
                     &Iter. & Time(s) &Accu. & Err. \\\hline
 51  & 14747& 4.79&0.99 &$1.59\times10^{-11}$ \\
501  & 26019& 14.52&0.98&$9.56\times10^{-12}$ \\
5,001  &30108&57.82&0.99 &$2.01\times10^{-11}$ \\
50,001  & 32387& 426.60&0.95 &$1.08\times10^{-11}$ \\
500,001  &30070& 6309.58&0.99&$4.49\times10^{-11}$\\
5,000,001  & 51609& 125869.02&0.97&$2.40\times10^{-10}$ \\
\hline
\multicolumn{5}{c}{The $4$-spectral radius of $6$-uniform $\beta$-stars ( $p<r-1$ )}\\
\end{tabulary}
\end{tabular}
\caption{The $p$-spectral radius of $r$-uniform $\beta$-stars.}
\label{TabBetaStarGre}
\end{table}

We  calculate $p$-spectral radii of  $\beta$-stars with various orders and edges and display the results in Table \ref{TabBetaStarGre}. The Err. column presents the relative error between our computational result and the corresponding exact result generated from Theorem \ref{BetaStarTheor}. It can be seen that all tests   succeed with high accuracy rates. Even the $3$-spectral radii and $4$-spectral radii of $\beta$-stars with millions of vertices are gained with high probability  and efficiency.

\subsection{Approximation of Lagrangians of  hypergraphs }

When $p=1,$ the $1$-spectral radius is also known as the Lagrangian of a hypergraph \eqref{lagrangian}. However,  $f(\x)$ is not  smooth at $\x$ who has  some  zero elements. We use $\lambda^{(p_\vartheta)}(G)$ to approximate $\lambda^{(1)}(G)$, with $p_\vartheta$ being denoted as
 \begin{equation}\label{p_j}
p_\vartheta=1+\frac{1}{2\vartheta+1}, \,\,\text{for} \,\, \vartheta=1,2,\ldots.
\end{equation}
Since $\lim_{\vartheta\rightarrow\infty}p_\vartheta=1,$ we have $\lim_{\vartheta\rightarrow\infty} \lambda^{(p_\vartheta)}(G)=\lambda^{(1)}(G)$ from Proposition \ref{PropoLimiPn}.
Therefore, we can use $p_\vartheta$-spectral radius to approximate the Lagrangian of a hypergraph.
The function $f_{p_\vartheta}(\x)$ is continuous and differentiable and CSRH method is  feasible for computing $p_\vartheta$-spectral radius  of a uniform hypergraph.
Let $\w$ be a vector such that its $i$th element being
$$w_i=x_i^{\frac{1}{2\vartheta+1}}, \quad \text{for}\,\, i= 1,\ldots,n.$$
Then  function $f_{p_\vartheta}(\x)=f_{p_\vartheta}(\w^{[2\vartheta+1]})$ is also a semialgebraic function and satisfies the {\L}ojasiewicz inequality \eqref{KL property}. Therefore, the conclusions in Section 4 hold for $p_\vartheta$ in \eqref{p_j}.

In this subsection, we show the results of CSRH method approximating Lagrangian of a hypergraph. First we give an example to demonstrate that the CSRH method  is competent to compute the $p$-spectral radius of a uniform hypergraph when $p$ is a fraction in \eqref{p_j}. Next,  the numerical results of  approximating the Lagrangians of complete hypergraphs by $p_\vartheta$-spectral radius are represented. The termination criteria of algorithms in the remaining part of this paper is set as $\|\g(\x)\| \leq 10^{-6}.$

\begin{table}[ht]
\footnotesize
\centering
\begin{tabulary}{1.0\textwidth}{l|llll} 
\hline
$p_n$& Iter.& Time(s)& Accu. & Err. \\
\hline
    $p_\vartheta=\frac{12}{7}$ & 3037 & 0.99 & 1.00&  $0.00$\\
   $p_\vartheta=\frac{14}{9}$ & 13271 & 17.88 & 1.00 &  $3.08\times10^{-16}$\\
   $p_\vartheta=\frac{10}{7}$ & 51018&110.53&1.00&$1.85\times10^{-16}$\\
    $p_\vartheta=\frac{4}{3}$& 84848 &88.85 & 1.00&$3.07\times10^{-14}$\\
  \hline
\end{tabulary}
\caption{$p_\vartheta$-spectral radius of   3-uniform $\beta$-star with 10 edges. }
\label{p_jBetaStarTab}
\end{table}
In Table \ref{p_jBetaStarTab}, we present the consequences of the $p_\vartheta$-spectral radius of a 3-uniform $\beta$-star with 10 edges, with $p_\vartheta$ being the fraction in the first column.  The true $p_\vartheta$-spectral radius can be acquired from Theorem \ref{BetaStarTheor}. All experiments produce the exact $p_\vartheta$-spectral radius with probability $1$ and the relative error between our numerical result and the theoretical value obtained from Theorem \ref{BetaStarTheor} is at most $3.07\times 10^{-14}.$

An $r$-uniform hypergraph  is said to be complete if it contains all possible edges when the number of its vertices is fixed. We use $C_n^r$ to denote a complete $r$-graph with $n$ vertices. Then  the 3-graph $C_4^3$ is actually  a tetrahedron with 6 edges.
The Lagrangian of  a complete uniform hypergraph
can be obtained directly from  Proposition \ref{PropLagran}.

\begin{figure}[h]
\centering
  \includegraphics[width=.5\textwidth]{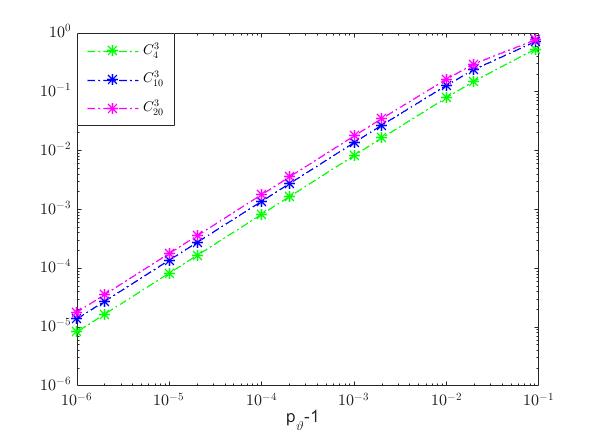}
\caption{Approximation of Lagrangian of complete hypergraphs.}\label{FigCompLagran}
\end{figure}
We compute different  $p_\vartheta$-spectral radii of 3 complete hypergraphs $C_4^3$, $C_{10}^3$ and $C_{20}^3$.  In Figure \ref{FigCompLagran}, the ordinate  reflects the error between the $p_\vartheta$-spectral radius and the true Lagrangian of the corresponding complete hypergraph which is obtained from the Proposition \ref{PropLagran}, while the abscissa means the value of $p_\vartheta-1.$  When $p_\vartheta$ approaches to $1,$ the $p_\vartheta$-spectral radius is  close to the  exact Lagrangian of the related hypergraph.

\section{Network analysis}

Not only the $p$-spectral radii, i.e., the optimal value of $f(\x)$ in \eqref{sperical constrained model}, but also the optimal point $\x$ in \eqref{sperical constrained model} characterize the structure of hypergraphs. Recall \eqref{Origi_pSpec} that an optimal point is called a $p$-optimal weighting.  The elements of the $p$-optimal weighting reflect the importance of the corresponding vertices in the hypergraph. Therefore, we may call the $i$th element of the $p$-optimal weighting the impact factor of the $i$th vertex. Different selections of the parameter $p$ provide different criteria of the importance of the vertices.
When $p$ is relatively large, the criterion tends to evaluate the importance of vertices more individually. When $p$ is relatively small, the ranking result demonstrates the significance of groups of vertices. In this section, we compute each $p$-spectral radius  10 times and choose the vector corresponding to the largest $f(\x)$ value as the $p$-optimal weighting.

 \subsection{A toy problem}

\begin{figure}[!tb]
\begin{center}
  \includegraphics[width=.7\textwidth]{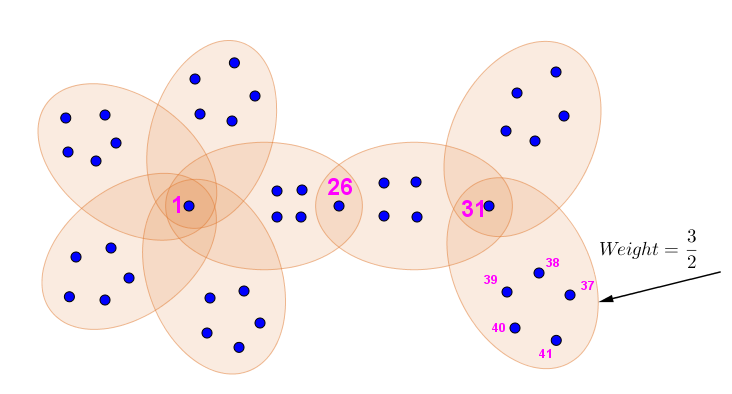}
  \caption{A 6-uniform hypergraph.}\label{BetaStarAppl}
\end{center}
\end{figure}

  We first employ a toy problem to illustrate the impact of the selections of $p$. We construct a 6-uniform weighted hypergraph with 8 edges as in Figure \ref{BetaStarAppl}. The weights of all edges of this hypergraph are set as $1$, except the last one whose weight is $\frac{3}{2}$.
  Obviously from the hypergraph, the vertices numbered $1$, $31$, and $26$ are distinct from other vertices, and the edge $\{31,37,38,39,40,41\}$ is also distinct from other edges. In Table \ref{Ranking1}, we show the different ranking of vertices via different $p$-optimal weighting. The abbreviation Num. means the number of a vertex and Val. represents the  impact factors of the corresponding vertices.

  When $p=\frac{4}{3},$ the top $6$ vertices are in the  edge who has the only largest weight among all edges. From Table \ref{Ranking1}, we can see that the impact factor of the top $6$ vertices in the $\frac{4}{3}$-optimal weighting are much greater than others. In fact, the value of all impact factors, except those corresponding to the top $6$ vertices, are  less than $5\times 10^{-10},$  which means that the dominant vertices are the ones from the largest weighted edge and the others can be ignored. That is to say, the ranking in this case offers the most important group of the vertices. When $p=5$, the vertex numbered $26$ appears in the top 10 list and the difference among the top $10$ impact factors is not as great as that when $p=\frac{4}{3}$. When $p=16$, the top $3$ vertices are $1$, $31$, $26,$ and the impact factors of vertices that have same status in the hypergraph are  rather close to each other. Then, we believe that the ranking results of $16$-spectral radius reflects the significance of vertices individually.

\begin{table}
\footnotesize
\centering
\begin{tabulary}{1.0\textwidth}{c|ll|ll|ll}
  \hline
\multirow{2}{3em}{Ranking} & \multicolumn{2}{c|}{$p=\frac{4}{3}$} & \multicolumn{2}{c|}{$p=5$} & \multicolumn{2}{c}{$p=16$}\\
\cline{2-7}
                       & Num. & Val. & Num. & Val. & Num. & Val. \\
\hline
 1 & 39 &  0.4082483175 & 41 & 0.4081204985 & 1&   0.1709715830  \\
 2 & 38 & 0.4082482858 & 39 & 0.4081204985 & 31 & 0.1678396311 \\
 3 & 31 & 0.4082482855 & 31 & 0.4081204983 & 26 &  0.1618288319  \\
 4 & 41 & 0.4082482854 & 38 & 0.4081204982 & 39 & 0.1600192388  \\
 5 & 40 & 0.4082482849 & 40& 0.4081204973 & 38 &  0.1600192387  \\
 6 & 37 & 0.4082482834 & 37 & 0.4081204958 & 41 & 0.1600192387  \\
 7 & 24 &  0.0000000000 & 28 & 0.0073198868 & 40 & 0.1600192386  \\
 8 & 34 &  0.0000000000 & 30 & 0.0073192175 & 37 & 0.1600192385 \\
 9 & 23 &  0.0000000000 & 26 & 0.0073061265 & 23 &  0.1550865094  \\
 10 & 3 &  0.0000000000 & 29 & 0.0071906282 & 22 & 0.1550865094  \\

\hline
\end{tabulary}
\caption{Top ten vertices in Figure \ref{BetaStarAppl}. }
\label{Ranking1}
\end{table}

\subsection{Author ranking}
\citeauthor{ng2011multirank}  in \cite{ng2011multirank} collected publication information  from DBLP\footnote{http://www.informatik.uni-trier.de/ ley/db/} and gave different rankings of the authors according to different factors, such as citations of authors, category concepts, collaborations, and papers.
In this subsection, we use the same data set in \cite{ng2011multirank} and rank the authors based on their collaborations.\footnote{We would like to thank Dr. Xutao Li for providing the database.}

We construct a weighted 3-uniform multi-hypergraph $G_A$ with $1,243,443$ edges to store the cooperation information. The vertex set is composed of  numbers of the 10305 authors and each edge has 3 vertices indicating that these three authors have cooperations under a same topic.   The weight of an edge is decided by the collaboration times among the three authors in this edge. The adjacency tensor of this multi-hypergraph $G_A$ is a sparse tensor with $1.17\%$ nonzero entries.

The example in Subsection $6.1$ shows that we can obtain the ranking score from different viewpoints by computing  different $p$-optimal weighting. Therefore, we compute  $2$-optimal weighting and $12$-optimal weighting of $G_A$ to get the  author group ranking and the author ranking respectively. In Figure \ref{OptiPoi}(a), the stars stand for the $2$-optimal impact factors of  vertices of $G_A.$  Obviously, the majority elements of $2$-optimal weighting are extraordinarily close to zero and only dozens of corresponding stars  are above the horizontal line of $y=0.1.$  In fact, $97.2\%$ of the entries in the $2$-optimal weighting are less than $10^{-3}$ and the elements that are greater than $0.1$ occupy only $1.8\%.$  On the other hand,  the largest impact factor reaches to $0.4481$ and the upper stars are considerably larger than others. It means that the $2$-optimal weighting is dominated by a small proportion of its components and we regard these leading elements as a group. The top ten authors ranked according to the 2-optimal impact factor are presented in the second column of Table \ref{AuthorRank}. The average collaboration times of each two authors among these top ten authors are $8.533,$ which is far larger than $9.76\times 10^{-4}$,  the average collaboration times of each two authors among the whole $10305 $ authors. Since these top ten authors have intimate cooperation, it is rational to consider them as a group and interpret the ranking in the second column as the most powerful group.

Stars in Figure \ref{OptiPoi}(b) are  the  $12$-optimal impact factors of vertices of $G_A.$ The distribution of these stars is totally different from the  ones in Figure \ref{OptiPoi}(a). It can be seen in Figure \ref{OptiPoi}(b) that the $12$-optimal  impact factors of the $10305$ authors are uniform and most of them are concentrated in the internal between $0.006$ and $0.014.$  Because in the original data set, the collaboration times of different authors are mostly one or two and we rank the authors based on their collaborations, the balance and concentration of the impact factors match up with the cooperation information.  The top ten authors generated via the $12$-optimal impact factors are listed in the third column of Table \ref{AuthorRank}. \citeauthor{ng2011multirank} also ranked the authors in the light of collaboration times and the influence of category concepts of their publications. We demonstrate  the top 10 authors of their experimental result \cite{ng2011multirank} in the MultiRank column in Table \ref{AuthorRank}. It can be seen that $6$ of the top $10$ authors in the  MultiRank are coincident with results of our $12$-optimal rank.

\begin{table}
\footnotesize
\centering
\begin{tabulary}{1.0\textwidth}{|c|l|l|l|}
  \hline
\multirow{2}{3em}{Ranking} & \multicolumn{3}{c|}{Author Name} \\
\cline{2-4}
                       & $p=2$ & $p=12$& MultiRank \\
\hline
 1 & Zheng Chen & Wei-Ying Ma & C. Lee Giles \\
 2 &Wei-Ying Ma &Zheng Chen &Philip S. Yu \\
 3 & Qiang Yang & Jiawei Han & Wei-Ying Ma\\
 4 & Jun Yan & Philip S. Yu   & Zheng Chen\\
 5 & Benyu Zhang & C. Lee Giles  & Jiawei Han\\
 6 & Hua-Jun Zeng & Jian Pei   & Christos Faloutsos\\
 7 & Weiguo Fan &  Christos Faloutsos  & Bing Liu\\
 8 & Wensi Xi &  Yong Yu     & Johannes Gehrke\\
 9 & Dou Shen &  Qiang Yang & Gerhard Weikum\\
 10 & Shuicheng Yan &  Ravi Kumar   & Elke A. Rundensteiner\\

\hline
\end{tabulary}
\caption{Top 10 authors.}
\label{AuthorRank}
\end{table}

\begin{figure}[!tb]
\begin{center}
\begin{tabular}{cccc}
  \includegraphics[width=.5\textwidth]{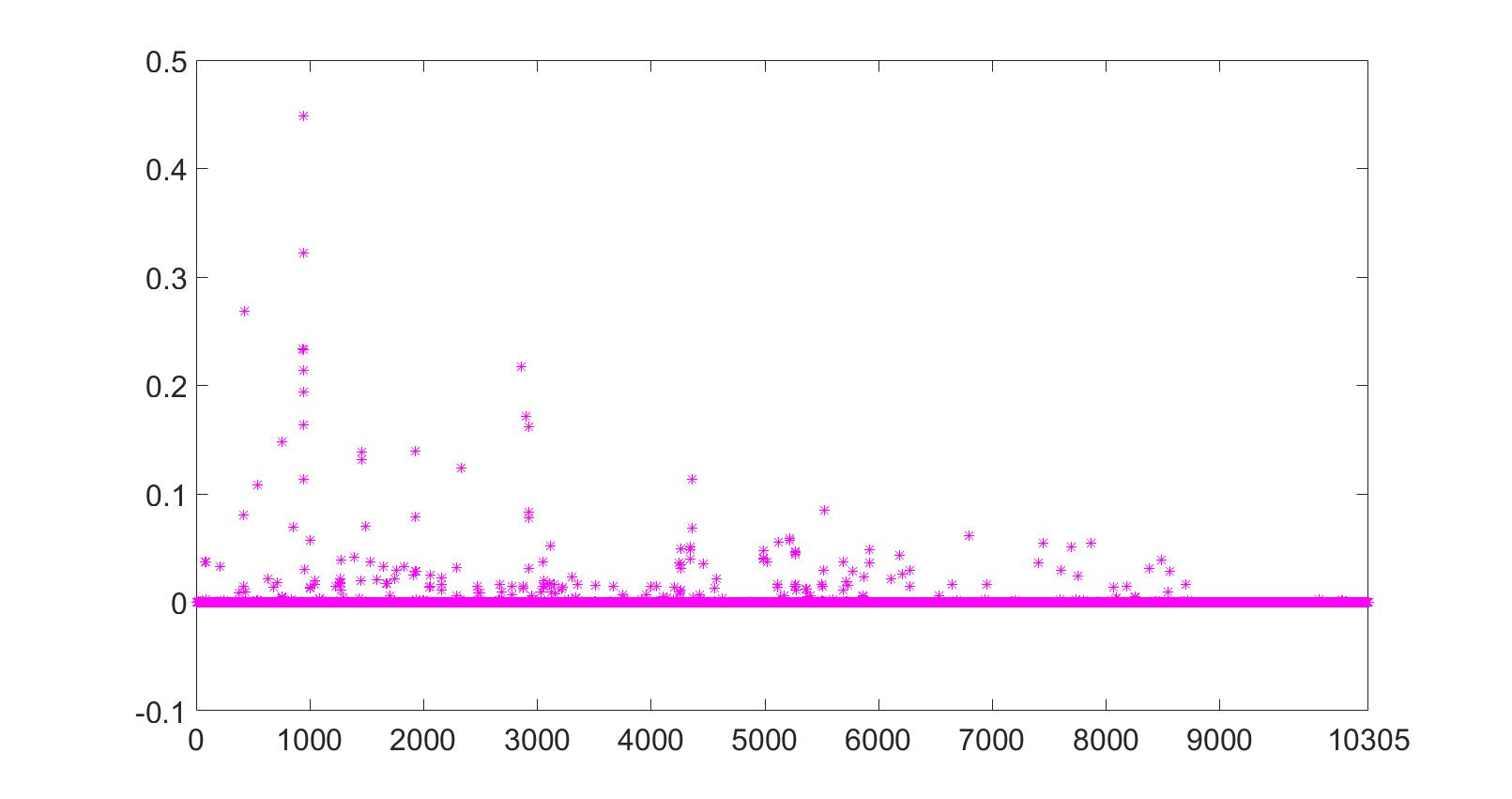} &
    \includegraphics[width=.5\textwidth]{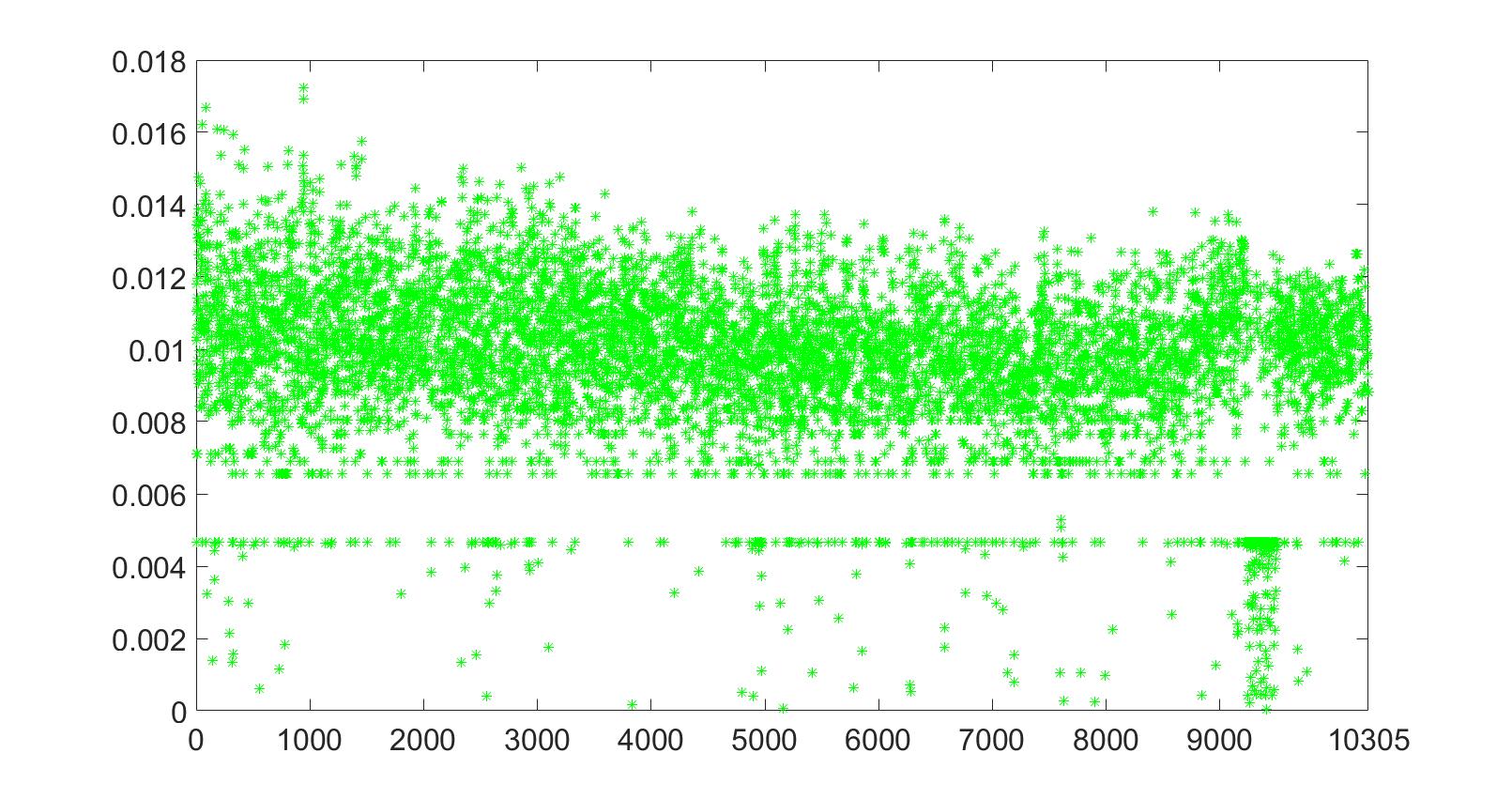}  \\
  (a) $p = 2$  & (b) $p = 12$
  \end{tabular}
\end{center}
  \caption{Optimal points.}\label{OptiPoi}
\end{figure}

\section{Conclusions}
We convert the $p$-norm constraint in $p$-spectral radius problem into an orthogonal constraint,  and propose a first order iterative algorithm CSRH for solving it. In this method, it is feasible to obtain a proper step length  to satisfy the Wolfe conditions under the curvilinear line search.  Convergence analysis shows that the CSRH method is globally convergent. The iterates converges to a $p$-optimal weighting.  Numerical experiments show that CSRH method is efficient and powerful.
In the author ranking application problem, we construct a weighted hypergraph with millions of edges. By computing $p$-spectral radius of this hypergraph, the  most influential cooperation group and the top ten ranked authors are presented.


\end{document}